\documentclass[12pt]{article}
\usepackage{authblk}

 \usepackage[utf8]{inputenc}
 \usepackage{scalerel}
 \usepackage{enumerate}
 \usepackage{authblk}
 \usepackage{wrapfig}

 \usepackage{amsmath,amscd}
 \usepackage{amsthm}
 \usepackage{amsfonts}
 \usepackage{color}

 \usepackage{amssymb}
 \usepackage[margin=1in]{geometry}

 \newtheorem{theorem}{Theorem}[section]
 \newtheorem{proposition}[theorem]{Proposition}
 \newtheorem{corollary}[theorem]{Corollary}
 \newtheorem{lemma}[theorem]{Lemma}
 \newtheorem{Remark}[theorem]{Remark}

 \theoremstyle{definition}
 \newtheorem{definition}{Definition}
 \newtheorem{example}[theorem]{Example}

\newcommand{\R}{
	\mathbb{R}
}

\newcommand{\smleq}{{\scaleobj{0.7}{\leq}}}

\newcommand{\dis}{\mathrm{dis}}

\newcommand{\hyp}{\mathrm{hyp}}
\newcommand{\poshyp}{\mathrm{hyp}^\smleq}
\newcommand{\gp}{\mathrm{g}}
\newcommand{\posgp}{\mathrm{g}^\smleq}

\newcommand{\pospaths}{\Gamma^\smleq}
\newcommand{\posm}{\mathrm{m}^\smleq}
\newcommand{\reeb}{\mathrm{R}}

\usepackage{chngcntr}
\usepackage{apptools}
\AtAppendix{\counterwithin{theorem}{section}}

\definecolor{darkblue}{rgb}{0.0, 0.0, 0.8}
\definecolor{darkred}{rgb}{0.8, 0.0, 0.0}
\definecolor{darkgreen}{rgb}{0.0, 0.8, 0.0}

\begin{document}

\title{Reeb Posets and Tree Approximations\footnote{This work was partially supported by NSF grants IIS-1422400 and  CCF-1526513.}}

\author[1]{Facundo M\'emoli}
\author[2]{Osman Berat Okutan}
	\affil[1]{Department of Mathematics and Department of Computer Science and Engineering, 
		The Ohio State University.
		\texttt{memoli@math.osu.edu}}
\affil[2]{Department of Mathematics, 
		The Ohio State University.
		\texttt{okutan.1@osu.edu}}

\maketitle
\begin{abstract}
A well known result in the analysis of finite metric spaces due to Gromov says that given any $(X,d_X)$ there exists a \emph{tree metric} $t_X$ on $X$ such that $\|d_X-t_X\|_\infty$ is bounded above by twice $\mathrm{hyp}(X)\cdot \log(2\,|X|)$. Here $\mathrm{hyp}(X)$ is the \emph{hyperbolicity} of $X$, a quantity that measures the \emph{treeness} of $4$-tuples of points in $X$. This bound is known to be asymptotically tight.

We improve this bound by restricting ourselves to metric spaces arising from filtered posets. By doing so we are able to replace the cardinality appearing in Gromov's bound by a certain poset theoretic invariant (the maximum length of fences in the poset) which can be much smaller thus significantly improving the approximation bound.

The setting of metric spaces arising from posets is rich: For example, save for the possible addition of new vertices, every finite metric graph can be induced from a filtered poset. Since every finite metric space can be isometrically embedded into a finite metric graph, our ideas are applicable to finite metric spaces as well.

At the core of our results lies the adaptation of the Reeb graph and Reeb tree constructions and the concept of hyperbolicity to the setting of posets, which we use to formulate and  prove a tree approximation result for any filtered poset.

\end{abstract}

\section{Introduction}
Trees, as combinatorial structures which model branching processes arise in a multitude of ways in computer science, for example as data structures that can help encoding the result of hierarchical clustering methods  \cite{j71}, or as structures encoding classification rules in decision trees \cite{duda12}. In biology, trees arise as phylogenetic trees \cite{ss03}, which help model evolutionary mechanisms. In computational geometry and data analysis trees appear for instance as contour/merge trees of functions defined on a manifold \cite{carr03,mbw13}.

From the standpoint of applications, datasets which can be associated a tree representation can be readily visualized. When a dataset does not directly lend itself to being represented as by a tree, motivated by the desire to visualize it, the question arises of what is the \emph{closest} tree to the given dataset. In this sense, one would then want to have (1) ways of quantifying the \emph{treeness} of data, (2) efficient methods for actually computing a tree that is (nearly) optimally close to the given dataset.

There are three different but related ways in which trees can be mathematically described. The first one is poset theoretic: a tree is a partially ordered set such that any two elements less than a given element are comparable, or in other words there is a unique way to go down the poset. The second is graph theoretic: a tree is a graph without loops. Finally, there is the metric way: a \emph{tree metric space} is a metric space which can be embedded in a metric tree (graph). This last description is the bridge between data analysis and combinatorics of trees. Through it, we can ask and eventually answer the following questions:

\begin{quote}\emph{How tree-like is a given metric data set?
How  does this treeness affect its geometric features? How can we obtain a tree which is close to a given dataset?}
\end{quote}

One measure of treeness of a metric space $(X,d_X)$ is given by the so called \emph{hyperbolicity constant}\footnote{Which is non-negative.} $\hyp(X,d_X)$ of $(X,d_X)$ \cite{bbi01} (see Section \ref{sec:hyp} for the definition). It is known that a metric space $(X,d_X)$ has $\hyp(X,d_X)=0$ if and only if there exists a tree metric space  $\mathcal{T}$ (i.e. a union of topological intervals without loops, endowed with the minimal path length distance) inside which $X$ can be isometrically embedded. Define $(X,d_X)$ to be a \emph{tree metric space} if and only if $\hyp(X,d_X)=0$, in which case we refer to $d_X$ as a \emph{tree metric} on $X$.

A natural question that ensues is whether the relaxed condition that $\hyp(X,d_X)$ be small (instead of $\hyp(X,d_X)=0)$, guarantees the existence of a tree metric on $X$ which is close to $d_X$. In this respect, in \cite{g87} Gromov shows that for each finite metric space $(X,d_X)$, there exists a tree metric $t_X$ on $X$ such that $$||d_X-t_X||_\infty \leq \Upsilon(X):=2 \, \hyp(X) \, \log(2|X|),$$  where $|X|$ is the cardinality of $X$. Despite the seemingly unsatisfactory fact that $\Phi(X)$ blows up with the cardinality of $X$ (unless $\hyp(X)=0$), it is known that this bound is asymptotically tight \cite{c16-nips}.  This suggests searching for alternative bounds which may perform better in more restricted scenarios.

\begin{wrapfigure}[10]{r}{0.3\textwidth}
		\vspace*{-0.3in}	
\includegraphics[width=0.3\textwidth]{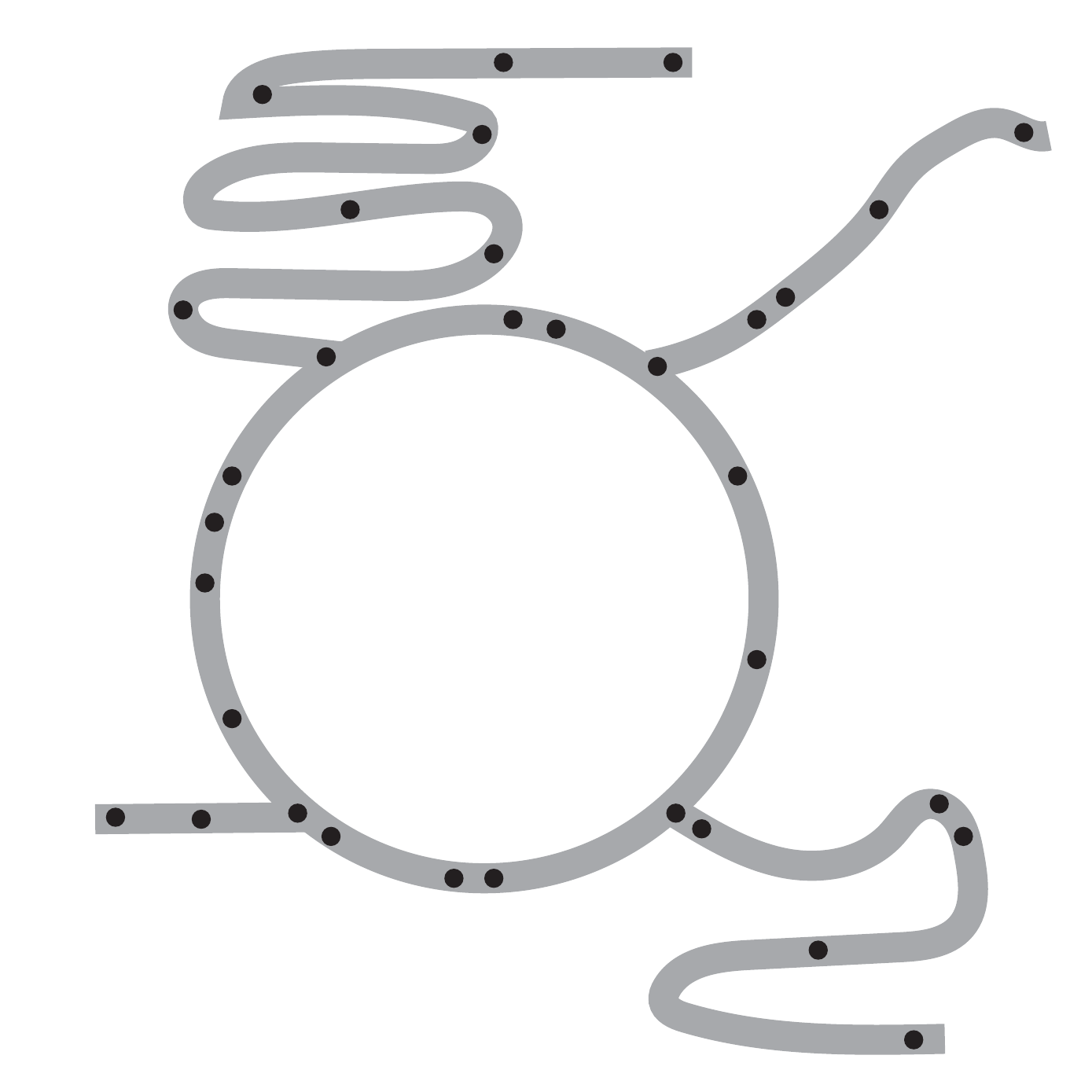}
	\end{wrapfigure} 

\paragraph{Contributions.} We refine Gromov's bound $\Upsilon(X)$ by identifying a quantity $\Phi(X)$ that is related to but often much smaller than Gromov's $\Upsilon(X)$. $\Phi(X)$ arises by  considering isometric embeddings of $X$ into a \emph{metric graph} $\mathcal{G}$. Note that this is always possible. Given one such $\mathcal{G}$ we then consider the product $$\phi(\mathcal{G}):=2\,\hyp(\mathcal{G})\,\log(4\beta_1(\mathcal{G})+4),$$ where $\beta_1(\mathcal{G})$ is the first Betti number of $\mathcal{G}$ (a notion of topological complexity). Finally, $\Phi(X)$ is defined  as the smallest possible value of $\phi(\mathcal{G})$ amongst all graphs $\mathcal{G}$ inside which $X$ can be isometrically embedded. 
We then obtain the following theorem:

 \begin{theorem}\label{thm:fms}
 For any finite metric space $(X,d_X)$ there exists a tree metric  $t_X$ such that  $$\|d_X-t_X\|_{\infty}\leq \Phi(X).$$
 \end{theorem}

\begin{example}
 Consider the case when $X_n$ is a finite sample consisting of $n$ points from a fixed metric graph $\mathcal{G}$ such as in the figure above. Assume that as $n$ grows the sample becomes denser and denser inside $\mathcal{G}$. In this case, as $\hyp(X_n)\simeq\hyp(G)$ since $\hyp$ is stable \cite{c16-nips}, we have $\Upsilon(X_n)\simeq 2\,\hyp(\mathcal{G})\,\log(2n)\rightarrow \infty$ as $n\rightarrow \infty$. On the other hand, $\Phi(X_n)$ is bounded by a constant/independent of $n$ (more precisely it will be bounded by $\phi(\mathcal{G})$). 
\end{example}

\begin{Remark}
Since Gromov's bound is known to be tight \cite{c16-nips} one would expect that there exists a sequence $(Z_n)$ of finite metric spaces such that both $\Upsilon(Z_n)$ and $\Phi(Z_n)$ have the same growth order. Such a construction is given in the Appendix, Section \ref{ex:same-rate}.
\end{Remark}

\smallskip\noindent
\textbf{The underlying idea: Reeb posets.} To obtain our bounds, we consider the case where $X$ is a metric space arising from a filtered poset. More precisely, given a poset $(X,\leq)$ with an order preserving filtration $f: X \to \R$, the filtration induces a distance $d_f$ on $X$ given by 
\begin{align*}
d_f(x,y):= \min \Bigg\{\sum_{i=1}^n | f(x_i)-f(x_{i-1})|:x_0=x,x_n=y, x_i \text{ is comparable with $x_{i+1}$ $\forall i$} \Bigg\}.
\end{align*}
A large class of metric spaces arise in this way. For example every metric graph, with possible addition of new vertices, can be realized this way. Hence, by embedding finite metric spaces into metric graphs, our methods can be applied to finite metric spaces. 

Given a a filtered poset $(X,\leq,f)$, we give a Reeb \cite{reeb} type construction to obtain a tree, which gives a metric $t_f$ on $X$. To obtain an upper bound for $\|d_f - t_f\|_\infty$, we define a filtered poset version $\poshyp_f$ of hyperbolicity and show that $||d_f-t_f||_\infty \leq  2\, \poshyp_f \, \log(2M_F), $
where $M_F$ is the poset theoretic constant given by the length of the largest fence in $(X,\leq)$. A fence is a finite chain of elements such that consecutive elements are comparable and non-consecutive elements are non-comparable. Note that the cardinality in the Gromov's result is replaced by $M_F$, which can be significantly smaller than the cardinality. 

\medskip
\noindent
\textbf{Organization of the paper.} In Section \ref{posets} we review and give some useful results about posets. In Section \ref{reeb-constructions} we introduce \textit{Reeb poset} and \textit{Reeb tree poset} constructions for filtered posets. In Section \ref{reeb-hyperbolicity} we introduce poset hyperbolicity for Reeb posets. In Section \ref{approximation} we consider tree metric approximations of Reeb posets. In Section \ref{application} we gave an application of Reeb poset constructions to finite graphs and metric spaces. In the Appendix, we give the necessary concepts and statements about metric spaces and graphs and we give an example where the growth rate of Gromov's bound is same as ours.
 
\section{Posets}\label{posets}
In this section we review some basic concepts for posets and give some results that we need later. For simplicity we are assuming that all posets we consider are finite and connected (i.e. each pair of points can be connected through a finite sequence $(x_0,\dots,x_n)$ of  points such that $x_i$ is comparable to $x_{i+1}$).
\begin{definition}[Covers and merging points]
Let $X$ be a poset. Given $x,y$ in $X$, we say that $x$ covers $y$ if $x>y$ and there is no $z$ such that $x>z>y$. A point is called a \textit{merging point} if it covers more than one elements. Given a point $x$, the number of points covered by $x$ is denoted by $\iota(x)$. Hence $x$ is a merging point if and only if $\iota(x)>1$. 
\end{definition}
\begin{lemma}\label{mergingpoint}
Let $X$ be a poset and $y,y'$ be non-comparable points in $X$. If there is a point $x$ such that $x > y,y'$, then there exists a merging point $x'$ such that $x \geq x' > y,y'$. 
\end{lemma}
\begin{proof}
Let $x'$ be a minimal element among all points satisfying $x\geq x'> y,y'$. We show that $x'$ is a merging point. Let $z$ be a maximal point satisfying $x'>z \geq y$ and $z'$ be maximal  satisfying $x' > z' \geq y'$. Note that x covers both $z,z'$ and $z \neq z'$ by the minimality of $x'$. 
\end{proof}

\begin{definition}[Chains and fences]
A totally ordered poset is called a \textit{chain}.  A fence is a poset whose elements can be numbered as $\{x_0,\dots,x_n\}$ so that $x_i$ is comparable to $x_{i-1}$ for each $i=1,\dots,n$ and no other two elements are comparable.  Note that a fence looks like a zigzag as its elements are ordered in the following fashion: $x_0<x_1>x_2<x_3>\dots$ or $x_0>x_1<x_2>x_3<\dots$. The \textit{length} of a chain or a fence is defined as the number of elements minus one.
\end{definition}

\begin{proposition}\label{fence-merging}
Let $X$ be a poset and $F$ be a fence with length $l$ in $X$. Then $X$ has at least $\lfloor \frac{l-1}{2} \rfloor$ merging points. 
\end{proposition}
\begin{proof}
Let us start with the case $l$ is even. Then $\lfloor \frac{l-1}{2} \rfloor=\frac{l-2}{2}$. By removing two endpoints if necessary, we get a fence of the form $x_0<x_1>x_2<\dots >x_{2k-2}<x_{2k-1}>x_{2k} $, where $l \leq 2k+2$. Let us show that $X$ has $k\geq \frac{l-2}{2}$ merging points. For $i=1,\dots,k$ let $y_i$ be a merging point such that $x_{2i-1} \geq y_i \geq x_{2i-2},x_{2i}$ whose existence is given by Lemma \ref{mergingpoint}. It is enough to show that $y_i$'s are distinct. Assume $i\leq j$ and $y_i=y_j$. Then $x_{2i-2} \leq y_i = y_j \leq x_{2j-1}$, so $2i-1 \geq 2j-1>2i-2$ and we have $i=j$. This completes $l$ is even case.

Now assume that $l$ is odd. Then by removing one of the endpoints, we get a fence of the form $x_0<x_1>x_2<\dots >x_{2k-2}<x_{2k-1}>x_{2k} $, where $l=2k+1$. Note that $k$ is exactly $\lfloor \frac{l-1}{2} \rfloor$ and by the analysis above $X$ has $k$ merging points.
\end{proof}

\begin{definition}[Covering graph]
The covering graph of a poset $X$ is the directed graph $(V,E)$ whose vertex set is $X$ and a directed edge is given by $(x,x')$ where $x'$ covers $x$.
\end{definition}
Recall that the first Betti number $\beta_1$ of a graph is defined as the minimal number of edges one needs to remove to obtain a tree.
\begin{proposition}\label{genus-merging}
Let $X$ be a poset with a smallest element $0$ and $G$ be the covering graph of $X$.  Then $\beta_1(G)=\sum_{x:\iota(x) \geq 1} (\iota(x)-1)$.
\end{proposition}
\begin{proof}
By Euler's formula, the first Betti number of a graph is equal to $1 + e - v$, where $e$ is the number of edges and $v$ is the number of vertices. Note that since edges in $G$ are given by the covering relations in $X$, the number of edges $e$ of $G$ is equal to $e=\sum_{x:\iota(x)\geq 1} \iota(x)$. Since the only vertex with $\iota(x)=0$ is $x=0$, we have $v-1=\sum_{\iota(x)\geq 1}1$. Hence 
$\beta_1(G)= 1 + e - v=e-(v-1)= \sum_{x:\iota(x) \geq 1} (\iota(x)-1). $
\end{proof}
\begin{corollary}\label{fence-genus}
Let $X$ be a poset with a smallest element $0$ and $\beta$ be the first Betti number of the covering graph of $X$. Then the length of a fence in $X$ is less than or equal to $2\beta+2$.
\end{corollary}
\begin{proof}
Let $F$ be a fence of length $l$ in $X$. Let $m$ be the number of merging points in $X$. Then by Proposition \ref{fence-merging} $\frac{l}{2}-1 \leq \lfloor \frac{l-1}{2} \rfloor \leq m$, so $l \leq 2m+2$. By Proposition \ref{genus-merging}, $m \leq \sum_{\iota(x)>1} (\iota(x)-1)= \beta$.
\end{proof}

\begin{definition}[Tree]
A (connected) poset $(T,\leq)$ is called a tree if for each $x$ in the the set of elements less than or equal to $x$ form a chain. In other words, non-comparable elements does not have a common upper bound.
\end{definition}

\begin{Remark}\label{tree-min}
A tree has a smallest element. 
\end{Remark}
\begin{proof}
Let $x$ be a minimal element. Let us show that it is the smallest element. Let $x'$ be a point in $T$ distinct from $x$.  By connectivity, there exists a minimal chain $(x_0,\dots,x_n)$ of elements so that $x_0=x,x_n=x'$. By minimality of the chain, this sequence is a fence. By minimality of $x_0=x$, $x_0 < x_1$. This implies $n=1$, since otherwise $x_1$ is an upper bound for non-comparable elements $x_0,x_2$. Hence $x=x_0 < x_1=x'$.
\end{proof}

\begin{proposition}\label{tree-merging}
A poset $T$ is a tree if and only if it does not contain any merging point.
\end{proposition}
\begin{proof}
``$\implies$'' Note that  if $x$ is a merging point then $x$ covers non-comparable elements, hence a tree does not contain merging points.

``$\impliedby$'' By Lemma \ref{mergingpoint}, if non-comparable elements have a common upper bound, then there exists a merging point. Hence if there are no merging points, then the set $\{x':x'\leq x \}$ is a chain for all $x$, hence $T$ is a tree. 
\end{proof}
The following proposition gives some other characterizations of tree posets.
\begin{proposition}\label{tree-char-fence}
Let $T$ be a poset with the minimal element $0$. Then the following are equivalent
\begin{enumerate}[(i)]
\item $T$ is a tree.
\item If $F$ is a fence in $T$, then the length of $F$ is less than or equal to two and if it is two then $F=x>y<z$.
\item The covering graph of $T$ is a tree.
\end{enumerate}
\end{proposition}
\begin{proof}
(i) $\implies$ (ii) A tree does not contain a fence of the form $x<y>z$. Any fence of length greater than two contains a sub-fence of the form $x<y>z$.

\noindent 
(ii) $\implies$ (iii) If $y$ is a merging point, then there are non-comparable elements $x,z$ such that $x<y>z$, hence $T$ does not contain any merging point. By Proposition \ref{genus-merging}, the genus of the covering graph is $0$, hence it is a tree.

\noindent 
(iii) $\implies$ (i) By Proposition \ref{genus-merging}, $T$ does not contain any merging points. By Proposition \ref{tree-merging}, $T$ is a tree.

\end{proof}    
 
 \section{Reeb constructions}\label{reeb-constructions}
In this section we generalize the definition of Reeb graphs \cite{reeb} to posets.

\subsection{Poset paths and length structures}
Two basic concepts used in defining Reeb graphs (as metric graphs) for topological spaces are those of \textit{paths} and \textit{length} \cite{b14}. We start by introducing these concepts in the poset setting.

\begin{definition}[Poset path]
Let $X$ be a poset and $x,y$ be points in $X$. A \textit{poset path} from $x$ to $y$  is an $n$-tuple $(x_0,\dots,x_n)$ of points of $X$ such that $x_0=x,x_n=y$ and $x_{i-1}$ is comparable with $x_i$ for $i=1,\dots,n$. We denote the set of all poset paths from $x$ to $y$ by $\pospaths(x,y)$. By $\pospaths$ we denote the union $\cup_{x,y \in X} \pospaths(x,y)$. Let us call a poset path \textit{simple} if $x_i \neq x_j$ for $i \neq j$. The \textit{image} of the path $(x_0,\dots,x_n)$ is $\{x_0,\dots,x_n \}.$
Note that a finite chain is the image of a simple path which is monotonous and fence is the image of a simple path $(x_0,\dots,x_n)$ where $x_i$ is not comparable with $x_j$ if $j \neq i-1,i,i+1$.
\end{definition}

Note that a poset path as defined above corresponds to an edge path in the comparability graph of the poset. Recall that the comparability graph of a poset is the graph whose set of vertices is the elements of the poset and the edges are given by comparable distinct vertices.

\begin{definition}[Inverse path]
Given a poset path $\gamma=(x_0,\dots,x_n)$, the \textit{inverse path} $\bar{\gamma}$ is defined as $\bar{\gamma}:=(x_n,\dots,x_0)$. 
\end{definition}

\begin{definition}[Concatenation of paths]
If $\gamma=(x_0,\dots,x_n)$ and $\gamma'=(y_0,\dots,y_m)$ are poset paths such that the terminal point $x_n$ of $\gamma$ is equal to the initial point $y_0$ of $\gamma'$, then we define the \textit{concatenation} of $\gamma,\gamma'$ by  $\gamma \cdot \gamma':=(x_0,\dots,x_n,y_1,\dots, y_m).$
\end{definition}

\begin{definition}[Length structure over a poset]
A \textit{length structure} $l$ over a poset $X$ assigns a non-negative real number to each poset path, which is additive under concatenation, invariant under path inversion and definite in the sense that non-constant paths have non-zero lengths.
\end{definition}

\begin{Remark}\label{metric-covering-graph}
If $X$ is a poset with a length structure, then its covering graph becomes a metric graph in a canonical way where the length of an edge is given by its length as a poset path.
\end{Remark}

\begin{definition}[Induced metric]
Given a length structure $l$ over a poset $X$, we define $d_l: X \times X \to [0,\infty]$ as $$d_l(x,y)=\inf_{\gamma \in \pospaths(x,y)} l(\gamma). $$
This is an extended metric over $X$ and called the metric \textit{induced} by $l$.
\end{definition}

\begin{definition}[Length minimizing paths]
Let $X$ be a poset with a length structure $l$. A poset path $\gamma \in \pospaths(x,y)$ is called \textit{length minimizing} if $d_l(x,y)=l(\gamma)$.
\end{definition}

We will later see that fences play an important role in minimization problems, including distance minimization.

\subsection{Reeb posets}
Let $X$ be a finite poset and $f: X \to (\R,\leq)$ be an order preserving function, i.e. $x \leq y$ implies that $f(x) \leq f(y)$. Let us define a relation $\sim$ on $X$ as follows: We say $x \sim y$ if there exists $\gamma \in \pospaths(x,y)$ such that $f$ is constant along $\gamma$, more precisely if $\gamma=(x_0,\dots,x_n)$, then $f(x_i)=f(x_j)$. This is an equivalence relation: reflexivity $x \sim x$ follows from the constant path $(x)$, symmetry follows from considering inverse path, and transitivity follows from concatenation. Let us denote the equivalence class of $x$ under this relation by $\tilde{x}$ and the quotient set $X/\sim$ by $\reeb_f(X)$. 

We define $\leq$ on $\reeb_f(X)$ as follows: $ \tilde{x} \leq \tilde{y}$ if there exists $\gamma \in \pospaths(x,y)$ such that $f$ is non-decreasing along $\gamma$. This is well defined since different representatives of the same equivalence class can be connected through an $f$-constant poset path.  Transitivity follows from  concatenation. Also note that if $\tilde{x} \leq \tilde{y}$ and $\tilde{y} \leq \tilde{y}$ then  $f$-nondecreasing paths connecting $x$ to $y$ and $y$ to $x$ have to be $f$-constant, so $\tilde{x}=\tilde{y}$. Therefore $(\reeb_f(X),\leq)$ is a poset. Also note that $f$ is still well defined on $\reeb_f(X)$ since $f$ is constant inside equivalence classes.

\begin{definition}[Reeb poset of $f: X \to \R$]
We call $(\reeb_f,\leq,f:\reeb_f \to \R)$ described above the \textit{Reeb poset} of the order preserving map $f: X \to \R$.
\end{definition}

\begin{Remark}\label{reeb-basic}
\begin{enumerate}[i)]
\item The quotient map $X \to \reeb_f(X)$ is order preserving.
\item $f:R_f \to \R$ is strictly order preserving.
\end{enumerate}
\end{Remark}
\begin{proof}
``i)'' If $x \leq y$, just take the $f$-nondecreasing path $(x,y)$.

``ii)'' Let $\tilde{x} <_f \tilde{y}$, then there exists a $f$-nondecrasing poset path $\gamma \in \pospaths(x,y)$. The poset path $\gamma$ is not $f$-constant as $\tilde{x} \neq \tilde{y}$. Therefore $f(\tilde{x})=f(x)<f(y)=f(\tilde{y})$. 
\end{proof}
\begin{Remark}\label{reeb-idempotent}
If $f: X \to \R$ is strictly order preserving, then $\reeb_f(X)=X$ as a poset.
\end{Remark}
\begin{proof}
Since $f$ is strict, no non-constant path is $f$-constant. Hence $\tilde{x}=\{x\}$. Therefore the order preserving quotient map $X \to \reeb_f(X)$ is an isomorphism.
\end{proof}
\begin{corollary}
$\reeb_f(\reeb_f(X))=\reeb_f(X)$.
\end{corollary}

\begin{Remark}
Let $R$ be a poset with an order preserving filtration $f:R \to \R$. $(R,f)$ is the Reeb poset of some order preserving map $f': (X,\leq) \to \R$ if and only if $f$ is strictly increasing.
\end{Remark}

Inspired by this remark, we give the following definition.
\begin{definition}[Reeb poset, Reeb tree poset]
A \textit{Reeb poset} $(R,f)$ is a poset $R$ with a strictly order preserving map $f: R \to \R$.  A Reeb poset is called a \textit{Reeb tree poset} if $R$ is a tree. (Recall that a poset is called a tree if for each element $x$ the elements less than or equal to $x$ form a chain.)
\end{definition}

\begin{definition}[Reeb metric]
A Reeb poset $(R,f)$ carries a canonical length structure $l_f$ defined as follows: For every poset path $\gamma = (x_0,\dots,x_n) \in \pospaths$
$$l_f(\gamma):=\sum_{i=1}^n|f(x_i)-f(x_{i-1})|.$$
Additivity and invariance with respect to inversion are obvious. Definiteness follows from the strictness of $f$. The metric structure induced by this length structure is denoted by $d_f$ and is called the \textit{Reeb metric} induced by $f$. 
\end{definition}

\begin{Remark}\label{distance-difference}
If $\gamma \in \pospaths(x,y)$, then $l_f(\gamma)\geq |f(x)-f(y)|$. Hence, $d_f(x,y) \geq |f(x)-f(y)|$. Therefore, if $x,y$ are comparable, then $d_f(x,y)=|f(x)-f(y)|$. Furthermore, for a Reeb graph $(R,f)$, $d_f(x,y)=|f(x)-f(y)|$ if and only if $x,y$ are comparable since the equality implies that there exists an $f$-increasing path between $x,y$ and in that case  $x,y$ are comparable by the equality $R_f(R)=R$.
\end{Remark}
\begin{Remark}\label{distancemin-fence}
The distance $d_f(x,y)$ can be realized as the $f$-length $l_f(\gamma)$ where $\gamma$ is a fence. This can be done by taking a length minimizing curve and removing points until one gets a fence (i.e. if $i+1<j$ and $x_i$ comparable to $x_j$ remove all points between $x_i,x_j$).  
\end{Remark}

\subsection{Reeb tree posets}

In this subsection we focus on the construction, characterization and properties of Reeb tree posets.

Let $X$ be a poset with an order preserving function $f: X \to \R$. Define an equivalence relation $\sim$ on $X$ as follows: $x \sim y$ if there exists $\gamma$ in $\pospaths(x,y)$ such that for each $x_i$ in $\gamma$ $f(x_i)\geq \max(f(x),f(y))$. Note that this implies that if $x \sim y$, then $f(x)=f(y)$. The fact that this is an equivalence relation follows from concatenation of paths. Let us denote the equivalence class of a point $x$ by $\tilde{x}$ and the quotient set by $T_f(X)$. Note that $f$ is still well defined on $T_f(X)$. Let us define a partial order $\leq$ on $T_f(X)$ as follows: $\tilde{x} \leq \tilde{y}$ if there exists $\gamma$ in $\pospaths(x,y)$ such that for each $x_i$ in $\gamma$, $f(x_i) \geq f(x)$. This is well defined since different representatives of an equivalence class can be connected by a poset path on which $f$ takes values greater than or equal to that of the representatives.   Let us show that this is a partial order. Reflexivity (i.e. $\tilde{x}\leq\tilde{x}$) follows from the constant path and transitivity (i.e. $\tilde{x} \leq \tilde{y} \leq \tilde{z} \implies \tilde{x} \leq \tilde{z}$)  follows from concatenation of paths. If $\tilde{x} \leq \tilde{y}$ then $f(x) \leq f(y)$, hence if $\tilde{x} \leq \tilde{y}$ and $\tilde{y} \leq \tilde{x}$, then $f(x)=f(y)$ thus the path giving $\tilde{x} \leq \tilde{y}$ also gives $\tilde{x}=\tilde{y}$. Note that this also shows that $f:T_f(X) \to \R$ is strictly order preserving. Note that $X \to T_f(X)$ is order preserving.

\begin{proposition}\label{reebtree}
$(T_f(X),\leq)$ is a tree.
\end{proposition}
\begin{proof}
Assume $\tilde{x},\tilde{x}' \leq \tilde{y}$. Let us show that $\tilde{x},\tilde{x}'$ are comparable.  WLOG assume that $f(x) \leq f(x')$. Let $\gamma$ (resp. $\gamma'$) be a poset path from $x$ (resp. $x'$) to $y$ giving $\tilde{x} \leq \tilde{y}$ (resp. $\tilde{x}' \leq \tilde{y}$). Then $\gamma \cdot \bar{\gamma}'$ is a path on which $f$ takes greater values than $f(x)$, hence $\tilde{x} \leq \tilde{x'}$.
\end{proof}

\begin{definition}[Reeb tree poset of $f: X \to \R$]
$(T_f(X),\leq,f)$ is called the \textit{Reeb tree poset} of $f: X \to \R$. We denote its Reeb metric by $t_f$. 
\end{definition}
\begin{Remark}
The minimal element of $T_f(X)$ is the equivalence class of $x$ where $f$ takes its minimal value.
\end{Remark}
\begin{Remark}\label{reeb-poset-tree-poset}
$T_f(R_f(X))=T_f(X)$ since the equivalence the relation used to define $R_f(X)$ is stronger than that of $T_f(X)$.
\end{Remark}
\begin{proposition}\label{reeb-tree-of-tree}
If $(T,\leq,f)$ is a Reeb tree poset, then $T_f(T)=T$.
\end{proposition}
\begin{proof}
As $T \to T_f(T)$ is order preserving, it is enough to show that $\tilde{x}=\tilde{y}$ if and only if $x=y$. Let $\gamma=(x_0,\dots,x_n)$ be a path from $x$ to $y$ giving $\tilde{x}=\tilde{y}$, i.e. $f(x_i)\geq f(x)=f(y)$ for each $i$. By removing elements from $\gamma$ if necessary, we can assume that $\gamma$ is a fence. By Proposition \ref{tree-char-fence}, $n$ is at most $2$. If $n=2$, then by Proposition \ref{tree-char-fence} $\gamma=(x >x_1<y)$, which is not possible by the strictness of $f$. If $n=1$, then $x,y$ are different but comparable, which is again not possible by the strictness of $f$. Hence $n=0$ and $x=y$.
\end{proof}

Now let us study properties of the Reeb tree metric $t_f$. Note that if $T$ is a tree poset and $x,y$ are points in $T$, than the intersection of the chains $\{z: z \leq x\}, \{z: z \leq x'\} $ is itself a chain hence it has a unique maximal element, which we denote by $p_{x,y}$.

\begin{proposition}\label{tree-metric}
Let $(T,f)$ be a Reeb tree poset. Then for all $x,y$ in $T$ $$t_f(x,y)=f(x)+f(y)-2f(p_{x,y}).$$
\end{proposition}
\begin{proof}
Without loss of generality $f(x)\leq f(y)$. If $x,y$ are comparable, then $p_{x,y}=x$ by the strictness of $f$. By Remark \ref{distance-difference} $t_f(x,y)=f(y)-f(x)=f(x)+f(y)-2f(p_{x,y})$. If $x,y$ are not comparable, then by Proposition \ref{tree-char-fence} $(x>p_x<y)$ is the only fence from $x$ to $y$. By Remarks \ref{distance-difference} and \ref{distancemin-fence}, $t_f(x,y)=f(x)-f(p_{x,y})+f(y)-f(p_{x,y})=f(x)+f(y)-2f(p_{x,y})$.
\end{proof}
Now, let us give a characterization of $f(p_{\tilde{x},\tilde{y}})$ for an order preserving map $f: X \to \R$. We first introduce the following definition:
\begin{definition}[Merge value $\posm_f$]
Let $f: X \to \R$ be an order preserving function. Define $$\posm_f(x,y):=\max_{\gamma \in \pospaths(x,y)} \min f \circ \gamma.$$
\end{definition}
\begin{proposition}\label{m-f}
Let $f: X \to \R$ be an order preserving map. Let the map $X \to T_f(X)$ be the map given by $x \mapsto \tilde{x}$. Let $p_{\tilde{x},\tilde{y}}$ be the maximal element in tree $T_f(X)$ which is less than or equal to both $\tilde{x},\tilde{y}$. Then
$$f(p_{\tilde{x},\tilde{y}})=\posm_f(x,y) .$$
We denote the right hand side of the equality above by $\posm_f(x,y)$.
\end{proposition}
\begin{proof}
Let $p_{x,y}$ be a point in the preimage of $p_{\tilde{x},\tilde{y}}$ in $X$. Let $\gamma$ be a path realizing $\posm_f(x,y)$ and $q_{x,y}$ be the point on it where $f$ takes its minimal value. Note that $ \tilde{q}_{x,y} \leq \tilde{x},\tilde{y}$, hence by the definition of $p_{\tilde{x},\tilde{y}}$ we have $\tilde{q}_{x,y} \leq p_{\tilde{x},\tilde{y}}$. So $f(q_{x,y})=f(\tilde{q}_{x,y})\leq f(p_{\tilde{x},\tilde{y}})$. Note that there are paths $\gamma \in \pospaths(x,p_{x,y}), \gamma' \in \pospaths (p_{x,y},y)$ such that on both curves $f$ takes values greater than or equal to $p_{x,y}$. Hence $f(q_{x,y}) \geq \min f \circ (\gamma\cdot\gamma')= f(p_{x,y})=f(p_{\tilde{x},\tilde{y}})$.
\end{proof}
As a corollary of Proposition \ref{tree-metric},\ref{m-f}, we have the following.

\begin{corollary}\label{reeb-tree-metric}
Let $f:X \to \R$ be an order preserving map. Then $$t_f(\tilde{x},\tilde{y})=f(x)+f(y)-2\,\posm_f(x,y).$$ 
\end{corollary}
\begin{Remark}\label{fence-mf}
$\posm_f(x,y)$ can be realized by a fence since if $\gamma$ is a poset path realizing $\posm_f(x,y)$, we can remove points until it becomes a fence without decreasing the minimal $f$-value.
\end{Remark}

\section{Hyperbolicity for Reeb posets}\label{reeb-hyperbolicity}

Metric hyperbolicity is a metric invariant which determines if a metric space is metric tree or not \cite{g87,bbi01}. In this section, we introduce a similar invariant which determines if a Reeb poset is a Reeb tree poset or not.

\begin{definition}(Gromov product for Reeb posets)
Let $(R,f)$ be a Reeb poset. We define the \textit{Gromov product} $$\posgp_f(x,y)=\frac{f(x)+f(y)-d_f(x,y)}{2}. $$
\end{definition}

\begin{definition}(Hyperbolicity for Reeb posets)
Let $(R,f)$ be a Reeb poset. We define the \textit{hyperbolicity} $\poshyp_f(R)$ as the minimal $\epsilon \geq 0$ such that for each $x,y,z$ in $X$, $$\posgp_f(x,z) \geq \min(\posgp_f(x,y),\posgp_f(y,z)) - \epsilon.$$
\end{definition}
The main statement of this section is the following.
\begin{proposition}\label{treeposet-hyp}
A Reeb poset $(T,f)$ is a Reeb tree if and only if $\poshyp_f(T)=0$. 
\end{proposition}
We give the proof after some remarks and lemmas.
\begin{Remark}\label{gp-comparable}
If $x \leq y$, then by Remark \ref{distance-difference} $g_f(x,y)=(f(x)+f(y)-(f(y)-f(x))/2=f(x).$
\end{Remark}
\begin{lemma}\label{gp-vs-f}
Let $(R,f)$ be a Reeb poset. Then, $g_f(x,y)\leq \min(f(x),f(y))$ for all $x,y$ in $R$. Equality happens if and only if $x,y$ are comparable.
\end{lemma}
\begin{proof}
Without loss of generality assume that $f(x)\leq f(y)$. Then,  by Remark \ref{distance-difference},  $$\posgp_f(x,y)-f(x)=(f(y)-f(x)-d_f(x,y))/2 \leq 0$$and equality happens if and only if $x,y$ are comparable.
\end{proof}

\begin{Remark}\label{treeposet-product}
If $(T,f)$ is a tree poset and $p_{x,y}$ is the maximal point which is smaller than both $x,y$. Then, by Proposition \ref{tree-metric}, $$ \posgp_f(x,y) =\frac{1}{2}(f(x)+f(y) - (f(x)+f(y)-2(p_{x,y})) )=f(p_{x,y}).$$
\end{Remark}

\begin{proof}[Proof of Proposition \ref{treeposet-hyp}]
``$\implies$'' Let $x,y,z$ be points in $T$. Then $p_{x,y},p_{y,z}$ are comparable since they are both less than $y$ and furthermore $p_{x,z} \geq \min(p_{x,y},p_{y,z})$ since that minimum is less than or equal to both $x$ and $z$. Therefore,  by Remark \ref{treeposet-product}
\[\posgp_f(x,z)=f(p_{x,z}) \geq \min(f(p_{x,y}),f(p_{y,z}))=\min(\posgp_f(x,y),\posgp_f(y,z)). \]
``$\impliedby$''  Assume $y\geq x,z$. Let us show that $x,z$ are comparable. By Lemma \ref{gp-vs-f} and Remark \ref{gp-comparable}, we have
$\min(f(x),f(z)) \geq \posgp_f(x,z) \geq \min(\posgp_f(x,y),\posgp_f(y,z))= \min(f(x),f(z)). $
Hence $\posgp_f(x,z)=\min(f(x),f(z))$ and by Lemma \ref{gp-vs-f} $x,z$ are comparable.

\end{proof}

\section{Approximation}\label{approximation}

In this section we consider tree approximations of Reeb posets. Our approximation result includes the following poset invariant.

\begin{definition}[Maximal fence length $M_F$]
Given a poset $X$, we define $$M_F(X):= \max_{F \text{ a fence in }X} |F|-1,$$
where $|F|$ is the number of elements in $F$. 
\end{definition}

\begin{theorem}\label{main}
Let $(R,f)$ be a Reeb poset. Let $\pi:R\rightarrow T_f(R)$ be the projection map. Then $$ |d_f(x,x')-t_f(\pi(x),\pi(x'))| \leq 2\,\log(2\,M_F(R))\,\poshyp_f(R).$$ Here we are considering logarithm base $2$.
\end{theorem}

We first prove following two lemmas.

\begin{lemma}\label{lipschitz}
Let $f: X \to \R$ be an order preserving map. Then $\posm_f(x,y) \geq \posgp_f(x,y).$
\end{lemma}
\begin{proof}
Let $\gamma$ in $\pospaths(x,y)$ be the path realizing $d_f(x,y)$. Let $z$ be the point on $\gamma$ where $f$ takes its minimal. Then by Remark \ref{distance-difference} we have $$d_f(x,y)=d_f(x,z)+d_f(z,y) \leq f(x)-f(z)+f(y)-f(z)=f(x)+f(y)-2f(z).$$ Hence we have
$\posgp_f(x,y)=(f(x)+f(y)-d_f(x,y))/2 \leq f(z) = \min f\circ \gamma \leq \posm_f(x,y). $
\end{proof}
\begin{lemma}\label{log-hyp}
Let $f: X \to \R$ be an order preserving map and $x_0,\dots,x_n$ be a family of elements in a poset $X$. Then 
$$\posgp_f(x_0,x_n) \geq \min_i \posgp_f(x_i,x_{i+1}) - \lceil \log n \rceil \, \poshyp_f(X).$$
\end{lemma}
\begin{proof}
Let us prove by induction on $n$. The case $n=1$ is trivial and $n=2$ follows from the definition of $\poshyp_f$. Now assume that the statement is true up to $n$ and $n > 3$. Let $k=\lceil n/2 \rceil$, then $k \geq n-k=\lfloor n/2 \rfloor$ and $\lceil \log 2k \rceil = \lceil \log n \rceil $. By the inductive hypothesis we have
\begin{align*} \posgp_f(x_0,x_n) &\geq \min(\posgp_f(x_0,x_k),\posgp_f(x_k,x_n)) - \poshyp_f(X)\\
&\geq \min_i \posgp_f(x_i,x_{i+1}) - (\lceil \log k \rceil + 1)\, \poshyp_f(X)\\
&= \min_i \posgp_f(x_i,x_{i+1}) - \lceil \log n \,\rceil \poshyp_f(X). 
\end{align*}
\end{proof}
Now, we can give the proof of Theorem \ref{main}. 
\begin{proof}[Proof of Theorem \ref{main}]
Let $\pi: R \to T_f(R)$ be the quotient map $x \mapsto \tilde{x}$. Since $\pi$ is surjective, it is enough to prove that the metric distortion $\dis(\pi)=\max_{x,y \in R}|d_f(x,y)-t_f(\tilde{x},\tilde{y})|$ satisfies $$\dis(\pi) \leq 2\log(2\,M_F(R))\,\poshyp_f(R).$$ By Corollary \ref{reeb-tree-metric}, we have
\[d_f(x,y)-t_f(x,y)=d_f(x,y) - \big(f(x)+f(y)-2\,\posm_f(x,y)\big)=2\big(\posm_f(x,y)-\posgp_f(x,y)\big). \]Hence, by Lemma \ref{lipschitz}, it is enough to show that for each $x,y$ in $X$, we have $$\posm_f(x,y)-\posgp_f(x,y) \leq \log(2\,M_F(R)) \,\poshyp_f(R). $$
By Remark \ref{fence-mf}, there exists a fence $(x=x_0,\dots,x_n=y)$ which realizes $\posm_f(x,y)$, i.e. $\posm_f(x,y)=\min_i f(x_i)$. Since $x_i,x_{i+1}$ are comparable, by Remark \ref{gp-comparable}, $\posgp_f(x_i,x_{i+1})=\min(f(x_i),f(x_{i+1})).$ Therefore, $\posm_f(x,y)=\min_i \posgp_f(x_i,x_{i+1}).$ Since $n \leq M_F(R)$, by Lemma \ref{log-hyp} 
$\posm_f(x,y) - \posgp_f(x,y) \leq \lceil \log n \, \rceil \poshyp_f(X) \leq \log(2\,M_F(R)) \, \poshyp_f(X). $
\end{proof}

\section{An application to metric graphs and finite metric spaces}\label{application}

In this section, we show how our poset theoretic ideas can be used to prove certain results for metric graphs and finite metric spaces. In particular we show that a metric graph naturally induces a poset with an order preserving filtration given by a distance function, and the induced metric from this filtration coincides with the original one. These observations makes our results for filtered posets applicable to graphs.

\begin{definition}[$p$-regularity]
Let $G=(V,E,l)$ be a simple metric graph with the length structure $l$ and $p$ be a vertex. We call $G$ $p$-\textit{regular} if it satisfies the following:
\begin{enumerate}[i)]
\item Each edge $(v,w)$ satisfies $l(v,w)=d_l(v,w)=|d_l(p,v)-d_l(p,w)|,$
\item Edges are the only length minimizing paths between their endpoints.
\end{enumerate}
\end{definition}
Note that by Proposition \ref{subdivision}, each metric graph can be extended by adding at most one vertex from the geometric realization of each edge so that the property ``i)" is satisfied. We can further extend the vertex set by adding the midpoints of edges which are not the only length minimizing path between their endpoints. After this extension, property ``ii)" is also satisfied.

In this section we prove the following:

\begin{theorem}\label{main-graph}
Let $G=(V,E,l)$ be a $p$-regular metric graph with the first Betti number $\beta$. Then there exists a tree metric space $(T,t_T)$ and a surjective map $\pi:X\rightarrow T$ such that 
$$\max_{x,x'\in V}|d_l(x,x')-t_T(\pi(x),\pi(x'))| \leq 2\,\log(4\beta+4)\,\hyp(V,d_l),$$
where $d_l$ is the metric induced by the length structure.
\end{theorem}
Note that any finite metric space $(X,d)$ can be isometrically embedded into a finite metric graph. The simplest example is the complete graph with the vertex set $X$ where the edge length is given by $d$. It is possible to obtain simpler embeddings \cite[Chapter 5.4]{ss03}. Here, we do not go into that path but assume that an embedding is already given.

Before providing the proof of Theorem \ref{main-graph} we go ahead and provide rthat of Theorem \ref{thm:fms}.
\begin{proof}[Proof of Theorem \ref{thm:fms}]
Let $p$ be a vertex of $G$. Without loss of generality we can assume that $G$ is $p$-regular since otherwise we can add some new vertices from its geometric realization to make it $p$-regular, as it is explained in the beginning of this section. Let $\pi: G \to (T,t_T)$ be the map given in Theorem \ref{main-graph}. Let $t: X \times X \rightarrow \R$ be the pseudo-metric given by $t(x,x')=t_T(\pi(x),\pi(x'))$. Then $t$ is a tree like metric on $X$ and the upper bound that we are trying to prove follows from \ref{main-graph}.
\end{proof}

\begin{Remark}
Note that for any finite metric space $X$ which can be isometrically embedded in the geometric realization of a metric graph $G$, the upper bound given in Theorem  \ref{thm:fms} still holds. This shows that $|X|$ can possibly be much larger than $4\beta+4$. In particular, the upper bound in Theorem \ref{thm:fms} can be much smaller than the one given by Gromov, i.e. $\log(2|X|) \, \hyp(X)$.
\end{Remark}

Through the following lemmas, we will carry this problem to the Reeb poset setting. Through this section assume that $G=(V,E,l)$ is a $p$-regular metric graph. 

\begin{lemma}\label{vertex-poset}
Define a relation $\leq_p$ on $V$ by $x \leq_p y$ if $d_l(p,y)-d_l(p,x)=d_l(x,y).$ Then $\leq_p$ is a partial order. 
\end{lemma}
\begin{proof}
Note that $x \leq_p x$ and $x \leq_p y$ implies that $d_l(p,x) \leq d_l(p,y).$ Hence, if $x \leq_p y, y \leq_p x$, then $d_l(x,y)=d_l(p,x)-d_l(p,y)=0$, which means that $x=y$. It remains to show the transitivity. Assume that $x \leq_p y \leq_p z$. Then we have
\begin{align*}
d_l(x,z) &\leq d_l(x,y)+d_l(y,z) \\
				&= d_l(p,y)-d_l(p,x)+d_l(p,z)-d_l(p,y)
                =d_l(p,z)-d_l(p,x) \leq d_l(x,z),
\end{align*}
so $d_l(x,z)=d_l(p,z)-d_l(p,x)$ which means that $x \leq_p z$.
\end{proof}
\begin{lemma}\label{vertex-reeb}
$(V,\leq_p,d_l(p,\cdot):V \to \R)$ is a Reeb poset.
\end{lemma}
\begin{proof}
Note that $d_l(p,\cdot)$ is strictly order preserving on $(V,\leq_p)$ since if $x <_p y$, then $d_l(p,y)-d_l(p,x)=d(x,y) > 0$.
\end{proof}
Note that the covering graph of a Reeb poset has a canonical metric graph structure given by $(v,w)\mapsto l_f(v,w)=|f(v)-f(w)|$.
\begin{lemma}\label{vertex-covering}
$G=(V,E,l)$ is the covering graph of the Reeb poset $(V,\leq_p,d_l(p,\cdot))$ as a metric graph (see Remark \ref{metric-covering-graph}). 
\end{lemma}
\begin{proof}
Let $G'=(V,E',l')$ be the covering graph of the Reeb poset $(V,\leq_p,d_l(p,\cdot))$. Let us show that $E=E'$ and $l=l'$. Note that if $(v,w)$ is an edge contained in $E \cap E'$, then by the property of $G$ with respect to $p$ described in Theorem \ref{main-graph}, $l(v,w)=|d(p,v)-d(p,w)|$, which is equal to $l'(v,w)$ by definition. Hence it remains to show that $E=E'$.

Let $v,w$ be a pair of distinct vertices. Without loss of generality we can assume that $d_l(p,v) \leq d_l(p,w)$. Let us show that $w$ covers $v$ if and only if $\{v,w\}$ is an edge. 

``$\implies$'' Take a length minimizing path $(v=v_0,\dots,v_n=w)$ of edges in $G$. Then we have
\begin{align*}
d_l(v,w) = \sum_i d_l(v_i,v_{i-1})
				 \geq \sum_i d_l(p,v_i)-d_l(p,v_{i-1})
                 =d_l(p,w)-d_l(p,w)=d_l(v,w),
\end{align*}
so for each $i$ we have $d_l(v_i,v_{i-1})=d_l(p,v_i)-d_l(p,v_{i-1})$, which means $v_i \geq v_{i-1}$. Since $w$ covers $v$, this means the path consists of two vertex. Since the path was arbitrary, this implies that $(v,w)$ is an edge.

``$\impliedby$'' By $p$-regularity $d_l(v,w)=d_l(p,w)-d_l(p,v)$, hence $v <_p w$. Let $(v_0,\dots,v_n)$ be a sequence of vertices such that $v_{i+1}$ covers $v_i$. Note that $w$ covers $v$ if and only if $n=1$. By the previous part $(v_{i-1},v_i)$ is an edge, hence $(v_0,\dots,v_n)$ is a path of edges in $G$. It is length minimizing since $d_l(v,w)=d_l(p,w)-d_l(p,v)=\sum_i d_l(p, v_i)-d_l(p,v_{i-1})=\sum_i d_l(v_{i-1},v_i). $ Since the edge is the only length minimizing path between its vertices, n=1. This completes the proof.
\end{proof}
\begin{lemma}\label{vertex-hyp}
$\poshyp_f(V,\leq_p)=\hyp_p(V,d_l)$, where $f:=d(p,\cdot): V \to \R$.
\end{lemma}
\begin{proof}
Note that by Lemma \ref{vertex-covering} $d_f=d_l$. Now the result follows since $$\posgp_f(v,w)=(f(v)+f(w)-d_f(v,w))/2=(d_l(p,v)+d_l(p,w)-d_l(v,w))/2=\gp_p(v,w).$$
\end{proof}
Now we can give the proof of Theorem \ref{main-graph}.
\begin{proof}[Proof of Theorem \ref{main-graph}]
Let $(T,t_f)$ be the Reeb poset tree of the Reeb poset $(V,\leq_p,f:=d_l(p,\cdot))$. Note that $(T,t_f)$ is a tree metric since it can be isometrically embedded into its covering graph, which is a metric tree by Proposition \ref{tree-char-fence}. By Lemma \ref{vertex-covering} and Corollary \ref{fence-genus} $d_f=d_l$, $M_F(V,\leq_p) \leq 2\beta+2$ and by Lemma \ref{vertex-hyp}, $\poshyp_f(V,\leq_p)=\hyp_p(V,d_l) \leq \hyp(V,d_l).$ . Now, the result follows from Theorem \ref{main}.
\end{proof}



\appendix
\section{Appendix}\label{appendix}

\subsection{Metric spaces}\label{sec:hyp}

For simplicity, we assume that all metric spaces we consider are finite.

\begin{definition}[Gromov product]
Let $(X,d)$ be a metric space and $p,x,y$ be points in $X$. The Gromov product $\gp_p(x,y)$ is defined by
\[\gp_p(x,y):=\frac{1}{2}(d(p,x)+d(p,y)-d(x,y)).\]
\end{definition}

\begin{definition}[Hyperbolicity]\label{def:hyp}
The $p$-hyperbolicity and hyperbolicity of $X$ can be respectively defined as follows:
\begin{align*}
\hyp_p(X)&:=\min \{\delta \geq 0: \gp_p(x,z) \geq \min(\gp_p(x,y),\gp_p(y,z)) - \delta \textrm{ for each } x,y,z \in X \}, \\
\hyp(X)&:= \max_{p \in X} \hyp_p(X).
\end{align*}
\end{definition}

The next results follow immediately.
\begin{proposition}
Let $X$ be a metric space and $p,q$ be any points in $X$. Then $$\hyp_p(X) \leq 2\, \hyp_q(X).$$
\end{proposition}

\begin{corollary}\label{zerohyp}
A metric space $X$ has 0 hyperbolicity if and only if there exists a point $p$ in $X$ such that $\hyp_p(X)=0$.
\end{corollary}

\begin{definition}[Tree metrics]
A finite tree metric is a finite metric space $T$ which can be isometrically embedded into a metric tree (i.e. a tree graph with a length structure). 
\end{definition}

\begin{proposition}\label{tree-hyp}
A finite metric space is a tree metric if and only if its hyperbolicity is $0$.
\end{proposition}
A proof of this can be found in \cite[Theorem 5.14]{ss03}.

\subsection{Graphs}

Introduce oriented graphs, trees, rooted trees, Betti number (and its equality to 1+ e -v)

For simplicity we assume that all graphs we  condiser are finite, simple and connected.
\begin{definition}[Graph]
A graph $G$ is a pair $(V,E)$ where $V$ is a (finite) set of \textit{vertices} and $E$ is a (finite) subset of two element subsets of $V$ and it is called the set of \textit{edges}. A directed edge is an edge with an ordering of its vertices. An \textit{ordered graph} is a graph with a selection of a (unique) direction for each edge.  
\end{definition}

\begin{definition}[Path]
A \textit{path} from a vertex $v$ to a vertex $w$ is a tuple $(v_0,\dots,v_n)$ of  vertices such that $v_0=v,v_n=w$ and $\{v_{i-1},v_i\}$ is an edge for each $i=1,\dots,n$.  A path in a directed graph is called \textit{directed} if the directions of its edges coincides with that of the graphs. A path is  called \textit{simple} if it consists of distinct vertices with the possibility of the exception $v_0=v_n$. If $(v_0,\dots,v_n)$ and $(w_0,\dots,w_m)$ are paths so that $v_n=w_0$, we define the concatenation $(v_0,\dots,v_n)\cdot(w_0,\dots,w_m)$ by $(v_0,\dots,v_n=w_0,\dots,w_m)$.
\end{definition}

\begin{definition}[Connected graph]
A graph is called \textit{connected} if there exists a path between each pair of vertices.
\end{definition}

\begin{definition}[Tree]
A connected graph $T$ is called a tree if  there exists a unique simple path between each pair of vertices. 
\end{definition}

\begin{definition}[First Betti number]
The \textit{first Betti number} $\beta_1(G)$ of a  connected graph  $G$ is defined as the minimal number of edges one needs to remove from $G$ to obtained a tree.
\end{definition}

The following proposition follows from the Euler formula.
\begin{proposition}\label{betti-formula}
Let $G$ be a connected graph, $V$ denote its number of vertices and $E$ denote its number of edges and $\beta:=\beta_1(G)$. Then $\beta=1-V+E$.
\end{proposition}

\begin{definition}[Metric graphs]
A \textit{metric graph} is a graph $G=(V,E)$ and an a function $l: E \to \R_{>0}$. Given an edge $e$ in a metric graph, we call $l(e)$ the \textit{length} of $e$. Each path $(v_0,\dots,v_n)$ in a metric graph $(G,l)$ can be assigned a length defined as $\sum_i l(v_{i-1},v_i)$. This induces a metric structure $d_l$ on the vertex set $V$ of $G$ through the length minimizing paths.
\end{definition}
\begin{Remark}\label{top-metric-graph}
The geometrical realization of a metric graph has a canonical length structure given by the isometric identification of the geometric realization of an edge $e$ with $[0,l(e)]$.  Furthermore, the inclusion of $(V,d_l)$ into this realization is an isometric embedding.
\end{Remark}
\begin{proposition}\label{subdivision}
Let $(G,l)$ be a metric graph and $p$ be a vertex. Adding at most one vertex from the geometric realization of each edge if necessary, we can guarantee that for each edge $\{v,w\}$ we have
$$l(v,w)=d_l(v,w)=|d_l(p,v)-d_l(p,w)|.$$
\end{proposition}
\begin{proof}
Without loss of generality we can assume that  $d_l(p,v) \leq d_l(p,w)$. Identify the geometric realization of the edge $(v,w)$ with $[0,l(v,w)]$. Let $t$ be the maximal element in that edge such that there is a length minimizing curve in the geometric realization of $G$ from $p$ to $t$ passing through $v$. Note that $t \geq 0$ and for any $t' > t$, all length minimizing curves from $p$ to $t'$ passes through $w$. Hence, there are length minimizing curves $\alpha,\alpha'$ from $p$ to $t$ such that $\alpha$ passes through $v$ and $\alpha'$ passes through $w$. Since these curves cover $[v,w]$ and every point in these curves other than $t$ has distance strictly smaller than $d_l(p,t)$ to $p$, then $t$ is the unique point on the edge where $d(p,\cdot)$ takes its maximum. Extend the vertex set adding all such points. Hence all local maximums of $d_l(p,\cdot)$ in the geometric realization of $G$ is contained in the vertex set. Furthermore, given an edge $(v,w)$ such that $d_l(p,v) \leq d_l(p,w)$, there exist a length minimizing curve in the geometric realization from $p$ to $w$ containing the edge, therefore $d_l(v,w)=l(v,w)=|d_l(p,v)-d_l(p,w)|$.
\end{proof}

\subsection{Example where $\Phi \sim \Upsilon$}\label{ex:same-rate}

\begin{figure}
\begin{center}
\includegraphics[width=0.5\textwidth]{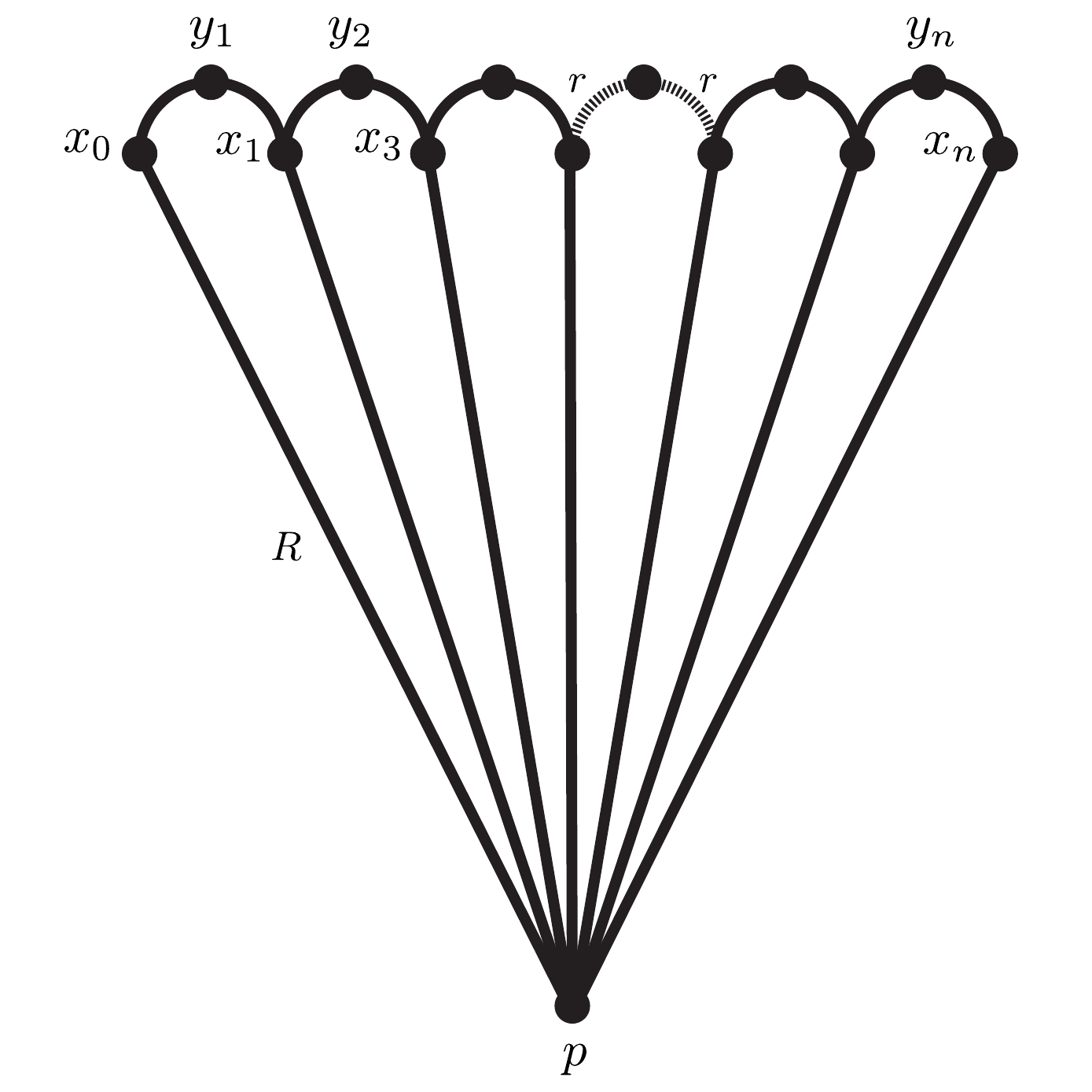}
\caption{Let $R\geq r >0$ and consider the metric graph from the figure. Let $Z_n$ be the finite subset $\{p,x_0,\ldots,x_n,y_1,\ldots,y_n\}.$ We show that $\Phi(Z_n) \sim 2 \log(4n)\,\hyp(Z_n)$ and $\Upsilon_p(Z_n) = 2 \log(4n+4)\,\hyp(Z_n)$.}\label{fig:ex:same-rate}
\end{center}
\end{figure}


Let $\mathcal{G}_n$ be the metric graph with the vertex set $Z_n$ as it is described in Figure \ref{fig:ex:same-rate}. The Gromov bound for $Z_n$ is $\Upsilon(Z_n) = 2\log(4n+4)\,\hyp(Z_n)$. 

Assume $Z_n$ is isometrically embedded in the geometric realization of metric a graph $G=(V,E,l)$. We can assume that $G$ is $p$-regular and $V$ contains $Z$. Consider the poset structure $\leq_p$ on $V$ described in Lemma \ref{vertex-poset}. Under this poset structure $(x_0,y_1,x_1,\dots,x_n)$ becomes a fence with length $2n$ (note that this is true independent of the embedding, since $\leq_p$ is completely determined by the metric). By Lemma \ref{vertex-covering} and Corollary \ref{fence-genus}, $\beta_1(G) \geq n-1$. Since $Z_n$ is a subspace of $G$, $\hyp(G)\geq \hyp(Z_n)$. Therefore, $\phi(G) \geq 2\, \log(4n) \, \hyp(Z_n)$. Since $G$ was arbitrary, we have $\Phi(Z_n) \geq 2\, \log(4n) \, \hyp(Z_n)$

If $R=r$, then one can show that $\hyp(\mathcal{G}_n)=\hyp(Z_n)$. Also note that $\beta(\mathcal{G}_n)=n$. In this case we get the upper bound $\Phi(Z_n) \leq \phi(\mathcal{G}_n) = 2 \, \log(4n+4) \hyp(Z_n)$. Therefore in this case $\Phi(Z_n),\Upsilon(Z_n)$ have the same growth rate.


\begin{thebibliography}{MBW13}

\bibitem[BBI]{bbi01}
Dmitri Burago, Yuri Burago, and Sergei Ivanov.
\newblock {\em A course in metric geometry}, volume~33.

\bibitem[BGW14]{b14}
U.~Bauer, X.~Ge, and Y.~Wang.
\newblock Measuring distance between {R}eeb graphs.
\newblock In {\em Proceedings of the Thirtieth Annual Symposium on
  Computational Geometry}, SOCG'14, pages 464--473, 2014.

\bibitem[CMS16]{c16-nips}
Samir Chowdhury, Facundo M{\'e}moli, and Zane~T Smith.
\newblock Improved error bounds for tree representations of metric spaces.
\newblock In {\em Advances in Neural Information Processing Systems (NIPS
  2016)}, pages 2838--2846, 2016.

\bibitem[CSA03]{carr03}
Hamish Carr, Jack Snoeyink, and Ulrike Axen.
\newblock Computing contour trees in all dimensions.
\newblock {\em Computational Geometry}, 24(2):75--94, 2003.

\bibitem[DHS12]{duda12}
Richard~O Duda, Peter~E Hart, and David~G Stork.
\newblock {\em Pattern classification}.
\newblock John Wiley \& Sons, 2012.

\bibitem[Gro]{g87}
Mikhael Gromov.
\newblock Hyperbolic groups.
\newblock {\em Essays in group theory}, 8(75-263):2.

\bibitem[JS71]{j71}
Nicholas Jardine and Robin Sibson.
\newblock Mathematical taxonomy.
\newblock {\em London etc.: John Wiley}, 1971.

\bibitem[MBW13]{mbw13}
Dmitriy Morozov, Kenes Beketayev, and Gunther Weber.
\newblock Interleaving distance between merge trees.
\newblock {\em Discrete and Computational Geometry}, 49:22--45, 2013.

\bibitem[Ree46]{reeb}
G.~Reeb.
\newblock Sur les points singuliers d'une forme de {P}faff compl\'{e}tement
  int\'{e}grable ou d'une fonction num\'{e}rique.
\newblock {\em Comptes Rendus de L'Acad\'{e}mie des Sciences}, 222:847--849,
  1946.

\bibitem[SS05]{ss03}
C.~Semple and M.A. Steel.
\newblock {\em Phylogenetics}.
\newblock Oxford lecture series in mathematics and its applications. Oxford
  University Press, 2005.

\end{thebibliography}
\end{document}